\documentclass[12pt,english,reqno]{amsart}
\usepackage{amsthm,amsmath,amssymb,amsfonts}
\usepackage[mathscr]{euscript}
\usepackage{mathtools}
\usepackage{mleftright}

\usepackage[
  pdfauthor={},
  pdfkeywords={},
  pdftitle={},
  pdfcreator={},
  pdfproducer={},
  linktocpage,colorlinks,bookmarksnumbered]{hyperref}

\theoremstyle{plain}
\newtheorem{theorem}{Theorem}
\newtheorem{lemma}{Lemma}
\newtheorem{corollary}{Corollary}

\theoremstyle{definition}

\newtheorem{example}{Example}
\newtheorem*{remark}{Remark}

\newcommand{\NN}{\mathbb{N}}
\newcommand{\ZZ}{\mathbb{Z}}
\newcommand{\QQ}{\mathbb{Q}}

\newcommand{\D}{\mathbf{D}}
\newcommand{\DD}{\mathbb{D}}
\newcommand{\DB}{\mathfrak{D}}
\newcommand{\MM}{\mathcal{M}}
\newcommand{\LB}{\mathcal{L}}

\newcommand{\SP}{\mathscr{S}}
\newcommand{\BP}{\mathscr{B}}

\DeclareMathOperator{\denom}{denom}
\DeclareMathOperator{\lcm}{lcm}
\DeclareMathOperator{\rad}{rad}
\DeclareMathOperator{\pval}{v\mspace{-1.5mu}}

\newcommand{\dmid}{\parallel}

\newcommand{\textdef}[1]{\emph{#1}}

\begin{document}

\title[The denominators of power sums]{The denominators of power sums of\\
arithmetic progressions}
\author{Bernd C. Kellner and Jonathan Sondow}
\address{G\"ottingen, Germany}
\email{bk@bernoulli.org}
\address{New York, USA}
\email{jsondow@alumni.princeton.edu}

\begin{abstract}
In a recent paper the authors studied the denominators of polynomials that
represent power sums by Bernoulli's formula.
Here we extend our results to power sums of arithmetic progressions.
In particular, we obtain a simple explicit criterion for integrality of the
coefficients of these polynomials. As applications, we obtain new results on
the sequence of denominators of the Bernoulli polynomials.
A consequence is that certain quotients of successive denominators are
infinitely often integers, which we characterize.
\end{abstract}

\maketitle


\section{Introduction}

For positive integers $n$ and $x$, define the \textdef{power sum}
$$
  S_n(x) := \sum_{k=0}^{x-1} k^n = 0^n + 1^n + \dots + (x-1)^n,
$$
and for integers $m \geq 1$ and $r \geq 0$ define the more general
\textdef{power sum of an arithmetic progression}
$$
  \SP_{m,r}^n(x) := \sum_{k=0}^{x-1} (km+r)^n = r^n + (m+r)^n + \dots + ((x-1)m+r)^n.
$$
In particular, we have $\SP_{1,0}^n(x) = S_n(x)$ and, more generally,
\begin{equation} \label{eq:sum-ap-0-Sn}
  \SP_{m,0}^n(x) = m^nS_n(x).
\end{equation}

Bazs\'o et al.\ \cite{Bazso&Mezo:2015,BPS:2012} considered the
generalized Bernoulli formula
\begin{equation} \label{eq:sum-ap-formula}
  \SP_{m,r}^n(x) = \frac{m^n}{n+1} \, \mleft(
    B_{n+1} \mleft( x + \frac{r}{m} \mright) - B_{n+1} \mleft( \frac{r}{m} \mright)
    \mright),
\end{equation}
where the $n$th Bernoulli polynomial $B_n(x)$ is defined by the series
$$
  \frac{te^{xt}}{e^t - 1} = \sum_{n=0}^{\infty} B_n(x) \frac{t^n}{n!}
  \quad (|t| < 2\pi)
$$
and is given by the formula
\begin{equation} \label{eq:bernpoly-def}
  B_n(x) = \sum_{k=0}^{n} \binom{n}{k} B_k \, x^{n-k},
\end{equation}
$B_k = B_k(0) \in \QQ$ being the $k$th Bernoulli number.
Thus, $\SP_{m,r}^n(x)$ is a polynomial in $x$ of degree $n+1$ with rational
coefficients.

\begin{remark}
Bazs\'o et al.\ required $r$ and $m$ to be coprime. However, since the forward
difference $\Delta B_n(x) := B_n(x+1) - B_n(x)$ equals $n x^{n-1}$
(cf.~\cite[Eq.~(5), p.~18]{Norlund:1924}), the telescoping sum of these
differences with $x = k + \frac{r}{m}$~implies \eqref{eq:sum-ap-formula}
at once for any $r/m \in \QQ$.
\end{remark}

For a polynomial \mbox{$f(x) \in \QQ[x]$}, define its \textdef{denominator},
denoted by $\denom \bigl( f(x) \bigr)$, to be the smallest $d \in \NN$
such that \mbox{$d \cdot f(x) \in \ZZ[x]$}.
This includes the usual definition of $\denom(q)$ for $q \in \QQ$.

In the classical case of Bernoulli's formula
$$
  S_n(x) = \frac{1}{n+1} \bigl( B_{n+1}(x) - B_{n+1} \bigr),
$$
the authors \cite[Thms.~1 and~2]{Kellner&Sondow:2017} determined the
denominator of the polynomial $S_n(x)$. From now on, let $p$ denote a prime.

\begin{theorem}[Kellner and Sondow \cite{Kellner&Sondow:2017}] \label{thm:denom-bernpoly}
For $n \geq 1$, denote
\begin{equation} \label{eq:Dn-def}
  \DD_n := \denom \bigl( B_n(x) - B_n \bigr).
\end{equation}
Then we have the relation
$$
  \denom \bigl( S_n(x) \bigr) = (n+1) \, \DD_{n+1}
$$
and the remarkable formula
\begin{equation} \label{eq:Dn-formula}
  \DD_n = \prod_{ \substack{
    p \, \leq \, \MM_n \\
    s_p(n) \, \geq \, p}} p
    \quad \text{with} \quad
  \MM_n :=
    \begin{cases}
      \, \dfrac{n+1}{2}, & \text{if $n$ is odd,} \\[0.5em]
      \, \dfrac{n+1}{3}, & \text{if $n$ is even,}
    \end{cases}
\end{equation}
where $s_p(n)$ denotes the sum of the base-$p$ digits of~$n$, as defined
in Section~\ref{sec:proofs2}. Moreover,
\begin{equation} \label{eq:Dn-odd}
  \DD_n \text{ is odd} \quad \iff \quad n = 2^k \;\; (k \geq 0).
\end{equation}
\end{theorem}

The first few values of $\DD_n$ are (see \cite[Seq. A195441]{OEIS})
$$
  \DD_n = 1, 1, 2, 1, 6, 2, 6, 3, 10, 2, 6, 2, 210, 30, 6, 3, 30, 10, 210, 42, 330, \dotsc.
$$

The sequence $(\DD_n)_{n \geq 1}$ and its properties will play a central role
in this paper. The denominators $\DD_n$ are involved in formulas for related
denominators in an essential way.
As implied by the product formula~\eqref{eq:Dn-formula}, it turns out that
the values of $\DD_n$ obey certain divisibility properties.
This culminates in the fact that certain quotients of successive denominators
$\DD_n$ are infinitely often integers, as we will see.

Here we extend Theorem~\ref{thm:denom-bernpoly} to the denominator of
$\SP_{m,r}^n(x)$, as follows.

\begin{theorem} \label{thm:main}
We have
\begin{equation} \label{eq:sum-ap-denom-formula}
  \denom \bigl( \SP_{m,r}^n(x) \bigr) =
    \frac{n+1}{\gcd(n+1,m^n)} \cdot \frac{\DD_{n+1}}{\gcd(\DD_{n+1},m)}.
\end{equation}
In particular, $\denom \bigl( \SP_{m,r}^n(x) \bigr)$ divides
$\denom \bigl( S_n(x) \bigr)$ and is independent of $r$.
Moreover, for any integers $r_1, r_2 \geq 0$,
$$
  \SP_{m,r_1}^n(x) - \SP_{m,r_2}^n(x) \in \ZZ[x].
$$
\end{theorem}

The next theorem shows exactly when $\SP_{m,r}^n(x)$ itself lies in $\ZZ[x]$.

\begin{theorem} \label{thm:main2}
For $n \geq 1$, denote
$$
  \DB_n := \denom \bigl( B_n(x) \bigr), \quad \D_n := \denom(B_n).
$$
Then we have the equivalence
$$
  \SP_{m,r}^n(x) \in \ZZ[x] \quad \iff \quad \DB_n \mid m
$$
as well as the equalities
\begin{equation} \label{eq:DB-formula}
  \DB_n = \lcm( \DD_n, \D_n )
\end{equation}
and
\begin{equation} \label{eq:DB-formula2}
  \DB_n = \lcm \bigl( \DD_{n+1}, \rad(n+1) \bigr),
\end{equation}
where $\rad(k) := \prod_{p \, \mid \, k} \, p$.
\end{theorem}

The first few values of $\DB_n$ and $\D_n$ are
(see \cite[Seqs. A144845 and A027642]{OEIS})
\begin{align*}
  \DB_n &= 2, 6, 2, 30, 6, 42, 6, 30, 10, 66, 6, 2730, 210, 30, 6, 510, 30, 3990, \dotsc, \\
  \D_n &= 2, 6, 1, 30, 1, 42, 1, 30, 1, 66, 1, 2730, 1, 6, 1, 510, 1, 798, 1, 330, \dotsc.
\end{align*}

\begin{example}
Set $m = \DB_n = 2, 6, 2, 30, 6$ for $n=1,2,3,4,5$,
respectively. Then certainly $\DB_n \mid m$, so $\ZZ[x]$ contains
the polynomials $\SP_{m,r}^n(x)$ with $r=0$ (which satisfy~\eqref{eq:sum-ap-0-Sn}):
\begin{align*}
  \SP_{2,0}^1(x)  &= x^2-x = 2\cdot\frac12 \, (x^2-x) = 2\cdot S_1(x),\\
  \SP_{6,0}^2(x)  &= 6 \, (2 x^3-3 x^2+x) = 6^2 \cdot \frac16 \, (2 x^3-3 x^2+x) = 6^2 \cdot S_2(x),\\
  \SP_{2,0}^3(x)  &= 2 \, (x^4-2 x^3+x^2) = 2^3 \cdot \frac14 \, (x^4-2 x^3+x^2) = 2^3 \cdot S_3(x),\\
\intertext{}
  \SP_{30,0}^4(x) &= 27000 \, (6 x^5-15 x^4+10 x^3-x) \\
  &= 30^4 \cdot \frac{1}{30} \, (6 x^5-15 x^4+10 x^3-x) = 30^4 \cdot S_4(x),\\
  \SP_{6,0}^5(x)  &= 648 \, (2 x^6-6 x^5+5 x^4-x^2) \\
  &= 6^5 \cdot \frac{1}{12}  \, (2 x^6-6 x^5+5 x^4-x^2) = 6^5 \cdot S_5(x)
\intertext{as well as those with $r=1$:}
  \SP_{2,1}^1(x)  &= x^2,\\
  \SP_{6,1}^2(x)  &= 12 x^3-12 x^2+x,\\
  \SP_{2,1}^3(x)  &= 2 x^4-x^2,\\
  \SP_{30,1}^4(x) &= 162000 x^5-378000 x^4+217800 x^3+24360 x^2-26159 x,\\
  \SP_{6,1}^5(x)  &= 1296 x^6-2592 x^5+540 x^4+1200 x^3-273 x^2-170 x.
\end{align*}
\end{example}

\begin{remark}
Bazs\'o and Mez\H{o} \cite[Eqs.~(7), (8) and Thm.~2, pp.~121--122]{Bazso&Mezo:2015}
defined a very complicated formula $F(n)$ in order to give a somewhat
tautological characterization of when $\SP_{m,r}^n(x) \in \ZZ[x]$.
With their formula they computed a few values of $F(n)$
that apparently equal $\DB_n$, but without recognizing this relation.
They were not aware of advanced results like those in our
Theorems~\ref{thm:main} and~\ref{thm:main2}.
\end{remark}

As an immediate by-``product'' of our theorems, we obtain a new product formula
for $\DB_n$ from \eqref{eq:DB-formula2} by applying Theorem~\ref{thm:denom-bernpoly}.
(Other explicit product formulas for this denominator, based on \eqref{eq:DB-formula},
were already given in \cite[Thm.~4]{Kellner&Sondow:2017}.)

\begin{corollary} \label{cor:Dn-prod}
For $n \geq 1$, the denominator of the $n$th Bernoulli polynomial equals
$$
  \DB_n =
  	\prod_{p \, \mid \, n+1} p
    \quad \times \quad \hspace{-0.6em}
    \prod_{
      \substack{
        p \, \nmid \, n+1 \\[0.2em]
        p \, \leq \, \MM_{n+1} \\[0.2em]
        s_p(n+1) \, \geq \, p}
    } \hspace{-0.2em} p.
$$
\end{corollary}

\begin{remark}
The first author \cite{Kellner:2017} has shown that the condition $s_p(n) \geq p$
is sufficient in \eqref{eq:Dn-formula} to define $\DD_n$ as a product over all primes:
\begin{equation} \label{eq:Dn-prod}
  \DD_n = \prod_{s_p(n) \, \geq \, p} p.
\end{equation}
(So one can remove the restrictions $p \leq \MM_n$ in \eqref{eq:Dn-formula} and
$p \leq \MM_{n+1}$ in Corollary~\ref{cor:Dn-prod}.)
Moreover, if $n+1$ is composite, then (see \cite[Thm.~1]{Kellner:2017})
\begin{equation} \label{eq:Dn-rad}
  \rad(n+1) \mid \DD_n.
\end{equation}
\end{remark}

Finally, we obtain new properties of $\DD_n$ and $\DB_n$.
\begin{corollary} \label{cor:Dn-relation}
The sequences $(\DD_n)_{n \geq 1}$ and $(\DB_n)_{n \geq 1}$ satisfy the following
conditions:
\begin{enumerate}
\itemsep5pt

\item We have the relations
  \begin{alignat*}{3}
    \DD_n &= \lcm \bigl( \DD_{n+1}, \rad(n+1) \bigr), & \quad & \text{if $n \geq 3$ is odd,} \\
    \DB_n &= \lcm \bigl( \DB_{n+1}, \rad(n+1) \bigr), && \text{if $n \geq 2$ is even.}
  \end{alignat*}

\item We have the divisibilities
  \begin{alignat*}{3}
    \DD_{n+1} &\mid \DD_n, & \quad & \text{if $n \geq 1$ is odd,} \\
    \DB_{n+1} &\mid \DB_n, && \text{if $n \geq 2$ is even.}
  \end{alignat*}

\end{enumerate}
\end{corollary}

\begin{theorem} \label{thm:Dn-ratio1}
For odd $n \geq 1$, the quotients $($see \cite[Seq. A286516]{OEIS}$)$
$$
  \frac{ \DD_{n}}{ \DD_{n+1}} = 1, 2, 3, 2, 5, 3, 7, 2, 3, 5, 11, 1, 13, 7, 15, 2, 17, 3, 19, 5, 7, \dotsc
$$
are odd, except that
$$
  \frac{ \DD_{n}}{ \DD_{n+1}} = 2 \quad
  \iff \quad n = 2^k - 1 \;\; (k \geq 2).
$$
Moreover, if $p$ is an odd prime and $n = 2^\ell p^k - 1$, then
\begin{alignat*}{3}
  \frac{ \DD_{n}}{ \DD_{n+1}} &\in \{ 1, p \} \quad
  && (k, \ell \geq 1),
\intertext{and more precisely,}
  \frac{ \DD_{n}}{ \DD_{n+1}} &= p
  && (k \geq 1, \, 1 \leq \ell < \log_2 p),
\intertext{while}
  \frac{ \DD_{n}}{ \DD_{n+1}} &= 1
  && (k \geq 1, \, \ell \geq \LB_p),
\end{alignat*}
where $\LB_p > \log_2 p$ is a constant depending on $p$.
\end{theorem}

\begin{theorem} \label{thm:Dn-ratio2}
For even $n \geq 2$, all terms are odd in the sequence\newline
$($see \cite[Seq.~A286517]{OEIS}$)$
$$
  \frac{ \DB_{n}}{ \DB_{n+1}} = 3, 5, 7, 3, 11, 13, 5, 17, 19, 7, 23, 5, 3, 29, 31, 11, 35, 37, \dotsc.
$$
In particular, if $p$ is an odd prime and $n = p^k - 1$, then
$$
  \frac{ \DB_{n}}{ \DB_{n+1}} = p \quad (k \geq 1).
$$
More generally, if $p \neq q$ are odd primes and $n = p^k q^\ell - 1$, then
$$
  \frac{ \DB_{n}}{ \DB_{n+1}} \in \{ 1, p, q, pq \}
  \quad (k, \ell \geq 1)
$$
with the following cases:
\begin{alignat*}{3}
  \frac{ \DB_{n}}{ \DB_{n+1}} &= p
  \quad && (k \geq \LB'_{p,q}, \, 1 \leq \ell < \log_q p), \\
  \frac{ \DB_{n}}{ \DB_{n+1}} &= q
  \quad && (1 \leq k < \log_p q, \, \ell \geq \LB''_{p,q}), \\
  \frac{ \DB_{n}}{ \DB_{n+1}} &= 1
  \quad && (k \geq \LB'_{p,q}, \, \ell \geq \LB''_{p,q}),
\end{alignat*}
where $\LB'_{p,q} > \log_p q$ and $\LB''_{p,q} > \log_q p$ are constants
depending on~$p$ and~$q$.
\end{theorem}

Theorems~\ref{thm:Dn-ratio1} and~\ref{thm:Dn-ratio2} immediately imply
the following result.

\begin{corollary}
Statements (i), (ii) (respectively, (iii), (iv)) below hold for infinitely
many odd (respectively, even) values of $n$:
\begin{enumerate}
\item $\DD_n / \DD_{n+1} = p$ for a given prime $p \geq 2$.
\item $\DD_n = \DD_{n+1}$.
\item $\DB_n / \DB_{n+1} = p$ for a given prime $p \geq 3$.
\item $\DB_n = \DB_{n+1}$.
\end{enumerate}
\end{corollary}


\section{Preliminaries}

Let $\ZZ_p$ be the ring of $p$-adic integers, $\QQ_p$ be the field of $p$-adic
numbers, and $\pval_p(s)$ be the $p$-adic valuation of $s \in \QQ_p$
(see \cite[Chap.~1.5, pp.~36--37]{Robert:2000}).
If $s \in \ZZ$, then $p^e \dmid s$ means that $p^e \mid s$ but $p^{e+1} \nmid s$,
or equivalently, $e = \pval_p(s)$.

The Bernoulli numbers satisfy the following properties
(cf.~\cite[Chap. 9.5, pp.~63--68]{Cohen:2007}).
The first few nonzero values are
\begin{equation} \label{eq:Bn-values}
  B_0 = 1,\; B_1 = -\frac{1}{2},\; B_2 = \frac{1}{6},\; B_4 = -\frac{1}{30},\;
  B_6 = \frac{1}{42},
\end{equation}
while $B_n=0$ for odd $n \geq 3$. For even $n \geq 2$ the von Staudt--Clausen
theorem states that the denominator of $B_n$ equals
\begin{equation} \label{eq:Bn-denom}
  \D_n = \prod_{p-1 \, \mid \, n} p \quad (n \in 2\NN).
\end{equation}
Thus, all nonzero Bernoulli numbers have a squarefree denominator.
Moreover, for even $n \geq 2$ the $p$-adic valuation of the
\textdef{divided Bernoulli number} $B_n/n$ is
\begin{equation} \label{eq:Bn-n-pval}
  \pval_p \mleft( \frac{B_n}{n} \mright) = \left\{
    \begin{array}{rl}
      -(\pval_p(n)+1), & \hbox{if $p-1 \mid n$,} \\
      \geq 0, & \hbox{else.}
    \end{array}
  \right.
\end{equation}

Now let $m$, $n$, and $r$ be positive integers.
The Bernoulli polynomials satisfy as Appell polynomials the general relation
\begin{equation} \label{eq:bernpoly-add}
  B_n(x+y) = \sum_{k=0}^{n} \binom{n}{k} B_k(y) \, x^{n-k},
\end{equation}
of which \eqref{eq:bernpoly-def} is a special case, as well as the reflection
formula
\begin{equation} \label{eq:bernpoly-reflect}
  B_n(1-x) = (-1)^n B_n(x)
\end{equation}
(see \cite[Chap.~3.5, pp.~114--115]{Prasolov:2010}).
Further, denote by $\BP_{m,r}^n$ the number
\begin{equation} \label{eq:BP-def}
  \BP_{m,r}^n
    := m^n \left( B_n \mleft( \frac{r}{m} \mright) - B_n \right)
    = \sum_{k=0}^{n-1} \binom{n}{k} B_k \, m^k r^{n-k}.
\end{equation}
Almkvist and Meurman \cite[Thm.~2, p.~104]{Almkvist&Meurman:1991} showed that
\begin{equation} \label{eq:BP-Z}
  \BP_{m,r}^n \in \ZZ.
\end{equation}
Actually, \eqref{eq:BP-Z} holds for all $r \in \ZZ$
(cf.~\cite[Thm.~9.5.29, pp.~70--71]{Cohen:2007}).
We also point out an analog to \eqref{eq:bernpoly-add} for $r_1, r_2 \in \ZZ$,
namely,
$$
  \BP_{m, r_1 + r_2}^n = \sum_{k=0}^{n} \binom{n}{k} \BP_{m,r_1}^k r_2^{n-k}
  \, + \, \BP_{m,r_2}^n.
$$

The integers $\BP_{m,r}^n$ satisfy a useful divisibility property, which we
need later on.
The following lemma is part of \cite[Thm.~11.4.12, pp.~327--329]{Cohen:2007},
but we give here a clearer and simpler proof.

\begin{lemma} \label{lem:BP-mod-power}
If $m,n \geq 1$, $r \in \ZZ$, a prime $p \nmid m$, and $0 \le e \le \pval_p(n)$,
then
\begin{equation} \label{eq:BP-mod-power}
  \BP_{m,r}^n \equiv 0 \pmod{p^e}.
\end{equation}
\end{lemma}

\begin{proof}
It suffices to prove the case $e = \pval_p(n)$. If $e = 0$, then we are trivially
done. So let $p^e \dmid n$ with
$$
  n > e = \pval_p(n) \geq 1.
$$
We split the proof into two cases as follows.

Case $p \mid r$:
From \eqref{eq:Bn-values} and \eqref{eq:BP-def} we deduce that
\begin{align*}
  \BP_{m,r}^n
    &= \sum_{k=0}^{n-1} \binom{n}{k} B_k \, m^k r^{n-k} \\
    &= r^n + \sum_{k=1}^{n-1} n \binom{n-1}{k-1} B_{n-k} \, m^{n-k} \frac{r^k}{k}.
\end{align*}
Since $p \mid r$, we have $\pval_p(r^n) \geq n$ and
$\pval_p ( r^k/k ) \geq 1$ for all $k \geq 1$.
If $B_{n-k} \neq 0$, then $\pval_p(B_{n-k}) \geq -1$, since the denominator is
squarefree. In this case we obtain
$$
  \pval_p \mleft( n \, B_{n-k} \, \frac{r^k}{k} \mright) \geq e.
$$
Considering all summands, we finally infer that \eqref{eq:BP-mod-power} holds.

Case $p \nmid r$:
Since $n \geq 2$, we have by~\eqref{eq:bernpoly-def}
and~\eqref{eq:bernpoly-reflect} that
$$
  B_n(1) - B_n = \sum_{k=0}^{n-1} \binom{n}{k} B_k = 0,
$$
which we use in the second step below.
Set $u := m/r \in \ZZ_p^\times$. As in the first case above, we derive that
\begin{align*}
  r^{-n} \, \BP_{m,r}^n
    &= \sum_{k=0}^{n-1} \binom{n}{k} B_k \, u^k \\
    &= \sum_{k=0}^{n-1} \binom{n}{k} B_k \cdot (u^k-1) \\
    &= - \frac{n}{2} (u-1) + \sum_{\substack{k=2\\2 \, \mid \, k}}^{n-1}
       n \binom{n-1}{k-1} \frac{B_k}{k} \, (u^k-1).
\end{align*}
In both cases $p=2$ and $p \geq 3$, we have
$$
  \frac{n}{2} (u-1) \equiv 0 \pmod{p^e}.
$$
Since \eqref{eq:Bn-n-pval} implies $B_k/k \in \ZZ_p$ if $k \geq 2$ is even
and $p-1 \nmid k$, we get
$$
  r^{-n} \, \BP_{m,r}^n
  \equiv \sum_{\substack{k=2\\2 \, \mid \, k\\p-1 \, \mid \, k}}^{n-1}
    n \binom{n-1}{k-1} \frac{B_k}{k} \, (u^k-1) \pmod{p^e}.
$$
Now fix one $k$ of the above sum. We then have the decomposition
$$
  k = a \, (p-1) \, p^t = a \, \varphi(p^{t+1}),
$$
where $p \nmid a$, $t = \pval_p(k)$, and $\varphi(\cdot)$ is Euler's totient
function. By assumption $u$ is a unit in $\ZZ_p$ and so is
$\hat{u} := u^a \in \ZZ_p^\times$. Euler--Fermat's theorem shows that
$$
  u^k \equiv \hat{u}^{\varphi(p^{t+1})} \equiv 1 \pmod{p^{t+1}}.
$$
Thus, $\pval_p(u^k-1) \geq t+1$. Since $\pval_p(B_k/k) = -(t+1)$ by
\eqref{eq:Bn-n-pval}, we achieve finally that
$$
  \pval_p \mleft( n \, \frac{B_k}{k} \, (u^k-1) \mright)
  \geq e - (t+1) + (t+1) = e,
$$
implying that
$$
  r^{-n} \, \BP_{m,r}^n \equiv 0 \pmod{p^e}
$$
and showing the result.
\end{proof}


\section{Proof of \texorpdfstring{Theorem~\ref{thm:main}}{Theorem~2}}

Before giving the proof of Theorem~\ref{thm:main},
we need several lemmas with some complementary results.
The next lemma easily shows a related partial result toward Theorem~\ref{thm:main},
while the full proof of this theorem requires much more effort.

\begin{lemma}
We have
\begin{align*}
  \denom \bigl( (n+1) \SP_{m,r}^n(x) \bigr)
  &= \denom \bigl( m^n (B_{n+1}(x) - B_{n+1}) \bigr) \\
  &= \frac{\DD_{n+1}}{\gcd(\DD_{n+1},m)}.
\end{align*}
\end{lemma}

\begin{proof}
By rewriting $\SP_{m,r}^n(x)$ as given in \eqref{eq:sum-ap-formula}, and using
\eqref{eq:bernpoly-def} and \eqref{eq:bernpoly-add}, we easily derive that
\begin{alignat}{1}
  (n+1) \SP_{m,r}^n(x) &= m^n \sum_{k=0}^{n} \binom{n+1}{k} \,
        B_k \mleft( \frac{r}{m} \mright) \, x^{n+1-k} \nonumber\\
    &=  m^n \sum_{k=0}^{n} \binom{n+1}{k} \left( B_k \mleft( \frac{r}{m} \mright)
        - B_k + B_k \right) x^{n+1-k}  \nonumber\\
    &= \sum_{k=0}^{n} \binom{n+1}{k} m^{n-k}
       \underbrace{m^k \left( B_k \mleft( \frac{r}{m} \mright) - B_k \right)}_{
       \BP_{m,r}^k \, \in \, \ZZ \text{ by } \eqref{eq:BP-Z}} x^{n+1-k}
       \label{eq:sumpart-1} \\
    & \quad + m^n \, \bigl( B_{n+1}(x) - B_{n+1} \bigr). \label{eq:sumpart-2}
\end{alignat}
By applying the simple observation that if $f(x) \in \ZZ[x]$ and
$g(x) \in \QQ[x]$, then
\begin{equation} \label{eq:denom-poly-add}
  \denom \bigl( f(x) + g(x) \bigr) = \denom \bigl( g(x) \bigr),
\end{equation}
we infer that
$$
  \denom \bigl( (n+1) \SP_{m,r}^n(x) \bigr)
  = \denom \bigl( m^n (B_{n+1}(x) - B_{n+1}) \bigr).
$$
Finally, from \eqref{eq:Dn-def} and~\eqref{eq:Dn-formula} we deduce that
$$
  \denom \bigl( m^n (B_{n+1}(x) - B_{n+1}) \bigr)
  = \frac{\DD_{n+1}}{\gcd(\DD_{n+1},m^n)}
  = \frac{\DD_{n+1}}{\gcd(\DD_{n+1},m)},
$$
the latter equation holding because $\DD_{n+1}$ is squarefree.
This completes the proof.
\end{proof}

\begin{lemma} \label{lem:coeff-cnk}
For positive integers $k \leq n$, define the rational number
$$
  c_{n,k} := \frac{1}{k} \binom{n}{k-1}.
$$
Then we have the following properties:
\begin{enumerate}

\item Symmetry:
  $$
    c_{n,k} = c_{n,n+1-k}.
  $$

\item Denominator:
  $$
    \denom(c_{n,k}) \mid \gcd(n+1, k), \quad \denom(c_{n,k}) \leq \frac{n+1}{2}.
  $$

\item Integrality:
  $$
    \text{If $k=1$ or $k=n$ or $n+1$ is prime, then $c_{n,k} \in \ZZ$}.
  $$

\end{enumerate}
\end{lemma}

\begin{proof}
We first observe that
\begin{equation} \label{eq:coeff-cnk}
  c_{n,k} = \frac{1}{k} \binom{n}{k-1} = \frac{1}{n+1} \binom{n+1}{k},
\end{equation}
which shows the symmetry in (i). From \eqref{eq:coeff-cnk} it also follows that
$\denom(c_{n,k})$ must divide both of the integers $n+1$ and $k$. Thus,
$$
  \denom(c_{n,k}) \mid \gcd(n+1, k).
$$
Since $k < n+1$, we then infer that
$\denom(c_{n,k}) \leq (n+1)/2$. This shows (ii).
If $k=1$ or $k=n$, then $c_{n,k} = 1$. If $n+1$ is prime, then
$\gcd(n+1, k) = 1$ as $k \leq n$, so $\denom(c_{n,k}) = 1$. This proves (iii).
\end{proof}

\begin{lemma} \label{lem:coeff-Z}
If $m,n,r \geq 1$ and $0 \leq k \leq n$, then
\begin{equation} \label{eq:coeff-Z}
  \frac{m^n}{n+1} \binom{n+1}{k} \left( B_k \mleft( \frac{r}{m} \mright)
  - B_k \right) \in \ZZ.
\end{equation}
\end{lemma}

\begin{proof}
If $k=0$, then the quantity in \eqref{eq:coeff-Z} vanishes by $B_0(x)-B_0 = 0$.
For $1 \leq k \leq n$, we can rewrite the quantity in \eqref{eq:coeff-Z} by
\eqref{eq:BP-def} and \eqref{eq:coeff-cnk} as
\begin{equation} \label{eq:triple-prod}
  c_{n,k} \, \times \, m^{n-k} \, \times \, \BP_{m,r}^k,
\end{equation}
where $\BP_{m,r}^k \in \ZZ$ by \eqref{eq:BP-Z} and
$c_{n,k} = \frac{1}{k} \binom{n}{k-1} \in \QQ$.
We have to show that \eqref{eq:triple-prod} lies in $\ZZ$.
If $k=1$ or $k=n$ or $n+1$ is prime, then $c_{n,k} \in \ZZ$ by
Lemma~\ref{lem:coeff-cnk}.
We can now assume that $n \geq 3$, $1 < k < n$, and $d := \denom(c_{n,k}) > 1$.
For each prime power divisor $p^e \dmid d$ we consider two cases, which
together imply the integrality of \eqref{eq:triple-prod}.

Case $p \nmid m$: Since $d \mid k$,
we have $p^e \mid \BP_{m,r}^k$ by Lemma~\ref{lem:BP-mod-power}.

Case $p \mid m$: We show that $p^e \mid m^{n-k}$, or equivalently,
\begin{equation} \label{eq:estimate-nek}
  n+1 > e+k.
\end{equation}
As $p^e \mid k$, by symmetry in Lemma~\ref{lem:coeff-cnk} we also have
\mbox{$p^e \mid n+1-k$}, so $e < n+1-k$ and \eqref{eq:estimate-nek} holds.
This completes the proof.
\end{proof}

\begin{proof}[Proof of Theorem~\ref{thm:main}]
To prove the last statement, it suffices to show that for $r \geq 0$
\begin{equation} \label{eq:sum-ap-diff-Z}
  \SP_{m,r}^n(x) - \SP_{m,0}^n(x) \in \ZZ[x].
\end{equation}
By \eqref{eq:sumpart-1} and \eqref{eq:sumpart-2} we have
\begin{equation} \label{eq:sum-ap-bernpoly}
  \SP_{m,r}^n(x) = \frac{m^n}{n+1} \bigl( B_{n+1}(x) - B_{n+1} \bigr) + h(x),
\end{equation}
where $(n+1)h(x) = f(x) \in \ZZ[x]$ as given by \eqref{eq:sumpart-1}.
By Lemma~\ref{lem:coeff-Z} it turns out that the coefficients of $h(x)$ are
already integral, and thus $h(x) \in \ZZ[x]$. Since by \eqref{eq:sum-ap-formula}
$$
  \SP_{m,0}^n(x) = \frac{m^n}{n+1} \bigl( B_{n+1}(x) - B_{n+1} \bigr),
$$
relation \eqref{eq:sum-ap-diff-Z} follows.

Applying the rule \eqref{eq:denom-poly-add} to \eqref{eq:sum-ap-bernpoly}
and using \eqref{eq:Dn-def} along with the fact that $B_{n+1}(x)$ is monic,
we then infer that
\begin{equation} \label{eq:sum-ap-denom-1}
  \denom \bigl( \SP_{m,r}^n(x) \bigr)
  = \denom \mleft( \frac{m^n}{(n+1)\DD_{n+1}} \mright).
\end{equation}
We have to show that \eqref{eq:sum-ap-denom-1} implies
\eqref{eq:sum-ap-denom-formula}.

In the following trivial cases we are done: Case $m=1$; cases $n=1,3$, since
$\DD_2 = \DD_4 = 1$; and case $n=2$, since $n+1=3$ and $\DD_3=2$.

So let $m \geq 2$ and $n \geq 4$. If a prime power $p^e \dmid n+1$, then $e < n$.
Consequently, we deduce that
\begin{equation} \label{eq:gcd-m-power}
  \gcd(n+1,m^n) = \gcd(n+1,m^{n-1}).
\end{equation}
Then by splitting $m^n$ into $m^{n-1} \cdot m$ in \eqref{eq:sum-ap-denom-1}
and applying \eqref{eq:gcd-m-power} and the fact that $\DD_{n+1}$ is squarefree,
we infer that \eqref{eq:sum-ap-denom-formula} holds.

Since $\denom \bigl( S_n(x) \bigr) = (n+1) \, \DD_{n+1}$ by
Theorem~\ref{thm:denom-bernpoly}, we see at once that
\eqref{eq:sum-ap-denom-formula} implies
$$
  \denom \bigl( \SP_{m,r}^n(x) \bigr) \mid \denom \bigl( S_n(x) \bigr).
$$
As a result of \eqref{eq:sum-ap-denom-1},
the denominator of $\SP_{m,r}^n(x)$ is independent of $r$.
This completes the proof of Theorem~\ref{thm:main}.
\end{proof}


\section{Proofs of \texorpdfstring{Theorem~\ref{thm:main2}}{Theorem~3}
and \texorpdfstring{Corollary~\ref{cor:Dn-relation}}{Corollary~2}} \label{sec:proofs2}

Before we give the proofs, we need some definitions and lemmas.
Recall that, given a prime $p$, any positive integer $n$ can be written
in base $p$ as a unique finite \textdef{$p$-adic expansion}
$$
  n = \alpha_0 + \alpha_1 \, p + \dotsb + \alpha_t \, p^t
  \quad (0 \leq \alpha_j \leq p-1).
$$
This expansion defines the \textdef{sum-of-digits function}
$$
  s_p(n) := \alpha_0 + \alpha_1 + \dotsb + \alpha_t,
$$
which satisfies the congruence
\begin{equation} \label{eq:sp-congr}
  s_p( n ) \equiv n \pmod{p-1}.
\end{equation}
Actually, these properties hold for any integer base $b \geq 2$ in place of
a prime~$p$.

The following lemma (see \cite[Chap.~5.3, p.~241]{Robert:2000})
shows the relation between $s_p(n)$ and $s_p(n+1)$.

\begin{lemma}
If $n \geq 1$ and $p$ is a prime, then
\begin{align}
  s_p(n+1) &= s_p(n) + 1 - (p-1) \, \pval_p(n+1). \nonumber \\
\shortintertext{In particular,}
  s_p(n+1) &\leq s_p(n) \hspace*{2.68em} \iff \quad \text{$p \mid n+1$},
  \nonumber \\
\shortintertext{while}
  s_p(n+1) &= s_p(n) + 1 \quad \iff \quad \text{$p \nmid n+1$}.
  \label{eq:sp-values-ndiv}
\end{align}
\end{lemma}

\begin{lemma} \label{lem:divisor-1}
If $n \geq 1$, then
$$
  \lcm( \DD_n, \D_n ) \mid \lcm \bigl( \DD_{n+1}, \rad(n+1) \bigr).
$$
\end{lemma}

\begin{proof}
Set $L_n := \lcm \bigl( \DD_{n+1}, \rad(n+1) \bigr)$. Since $\DD_n$ and $\D_n$
are squarefree by~\eqref{eq:Dn-formula} and~\eqref{eq:Bn-denom}, we show that
$p \mid \lcm( \DD_n, \D_n )$ implies $p \mid L_n$.
Moreover, since $\rad(n+1) \mid L_n$, we may assume that $p \nmid n+1$.

If $p \mid \DD_{n}$, then by \eqref{eq:Dn-formula} we have $s_p(n) \geq p$.
Applying \eqref{eq:sp-values-ndiv} followed by \eqref{eq:Dn-prod}, we obtain
$p \mid \DD_{n+1}$, and finally $p \mid L_n$.

Since $\D_1 = 2$ by~\eqref{eq:Bn-values} and $\D_n = 1$ for odd $n \geq 3$,
we have $\D_n \mid L_n$ for odd $n \geq 1$. So take $n \geq 2$ even.
If $p \mid \D_n$, then $p-1 \mid n$ by \eqref{eq:Bn-denom}, so also
$p-1 \mid s_p(n)$ by \eqref{eq:sp-congr}. Thus $s_p(n) \geq p-1$.
As $p \nmid n+1$ by assumption, \eqref{eq:sp-values-ndiv} implies
$s_p(n+1) \geq p$, so $p \mid \DD_{n+1}$ by \eqref{eq:Dn-prod}.
Finally $p \mid L_n$. This proves the lemma.
\end{proof}

\begin{lemma} \label{lem:divisor-2}
If $n \geq 1$, then
$$
  \lcm \bigl( \DD_{n+1}, \rad(n+1) \bigr) \mid \lcm( \DD_n, \D_n ).
$$
\end{lemma}

\begin{proof}
As $\DD_1 = \DD_2 = 1$, and $\D_1 = 2$ by~\eqref{eq:Bn-values}, the case
$n=1$ holds. So assume $n \geq 2$ and set $L_n := \lcm( \DD_n, \D_n )$.

If $n+1$ is not prime, then \eqref{eq:Dn-rad} implies $\rad(n+1) \mid L_n$.
Otherwise, $p = n+1 = \rad(n+1)$ is an odd prime and so $n$ is even.
By \eqref{eq:Bn-denom} we have $p \mid \D_n$, so $\rad(n+1) \mid L_n$.

It remains to show that $\DD_{n+1} \mid L_n$.
As $\DD_{n+1}$ is squarefree by \eqref{eq:Dn-formula},
it suffices to show for any prime $p \mid \DD_{n+1}$ that $p \mid L_n$.
By \eqref{eq:Dn-formula} again we have $s_p(n+1) \geq p$,
and as $\rad(n+1) \mid L_n$ we may assume that $p \nmid n+1$.
Then by \eqref{eq:sp-values-ndiv} we obtain $s_p(n) = s_p(n+1)-1 \geq p-1$.
If $s_p(n) \geq p$, then $p \mid \DD_n$ by \eqref{eq:Dn-prod}, so $p \mid L_n$.
Otherwise, $s_p(n) = p-1$ and so $p-1 \mid n$ by \eqref{eq:sp-congr}.
Moreover, $n$ must be even, as $n$ odd would imply $p=2$, contradicting
$p \nmid n+1$. Hence $p \mid \D_n$ by \eqref{eq:Bn-denom}, and finally
$p \mid L_n$. This completes the proof.
\end{proof}

\begin{proof}[Proof of Theorem~\ref{thm:main2}]
To show the equivalence, we have to prove that
$$
  \denom \bigl( \SP_{m,r}^n(x) \bigr) = 1
  \quad \iff \quad
  \DB_n \mid m.
$$
By \eqref{eq:sum-ap-denom-formula}, we have
$\denom \bigl( \SP_{m,r}^n(x) \bigr) = 1$ if and only if
$$
  \frac{n+1}{\gcd(n+1,m^n)} = \frac{\DD_{n+1}}{\gcd(\DD_{n+1},m)} = 1,
$$
which in turn is true if and only if $n+1 \mid m^n$ and $\DD_{n+1} \mid m$.
Moreover,
$$
  n+1 \mid m^n \quad \iff \quad \rad(n+1) \mid m.
$$
Indeed, $p \mid n+1 \mid m^n$ implies $p \mid m$, proving the ``$\Rightarrow$''
direction. Conversely, if $p \mid \rad(n+1) \mid m$, then $p^n \mid m^n$.
But $p^e \dmid n+1$ with $e \leq n$, so finally $n+1 \mid m^n$. It follows that
$$
  \denom \bigl( \SP_{m,r}^n(x) \bigr) = 1
  \quad \iff \quad
  \lcm \bigl( \DD_{n+1}, \rad(n+1) \bigr) \mid m.
$$
By Lemmas \ref{lem:divisor-1} and \ref{lem:divisor-2}, together with the proof
of \cite[Thm.~4]{Kellner&Sondow:2017}, we have
$$
  \lcm \bigl( \DD_{n+1}, \rad(n+1) \bigr) = \lcm( \DD_n, \D_n ) = \DB_n.
$$
This proves the theorem.
\end{proof}

\begin{proof}[Proof of Corollary~\ref{cor:Dn-relation}]
(i), (ii)
If $n \geq 3$ is odd, then $\D_n = 1$. Hence,
\eqref{eq:DB-formula} and \eqref{eq:DB-formula2} yield
$\DD_n = \lcm \bigl( \DD_{n+1}, \rad(n+1) \bigr)$.
Together with $\DD_1 = \DD_2 = 1$, this implies that
$\DD_{n+1} \mid \DD_{n}$ for all odd $n \geq 1$, as desired.

Similarly, for even $n \geq 2$, we have $\DB_{n+1} = \DD_{n+1}$ by
\eqref{eq:DB-formula}. Then \eqref{eq:DB-formula2} gives
$\DB_{n} = \lcm \bigl(\DB_{n+1}, \rad(n+1) \bigr)$, so $\DB_{n+1} \mid \DB_{n}$,
as claimed.
\end{proof}


\section{Proofs of \texorpdfstring{Theorems~\ref{thm:Dn-ratio1}}{Theorems~4}
and \texorpdfstring{\ref{thm:Dn-ratio2}}{5}}

Let $a, b \geq 2$ be integers that are multiplicatively independent,
that is, $a^e \neq b^f$ for all integers $e, f \geq 1$.
Senge and Straus \cite[Thm.~3]{Senge&Straus:1973} showed that
for a given constant $A$ the number of integers $n$ satisfying
$$
  s_a(n) + s_b(n) < A
$$
is finite. Steward \cite[Thm.~1, p.~64]{Stewart:1980} proved the effective
lower bound
\begin{equation} \label{eq:sum-digits}
  s_a(n) + s_b(n) > \frac{\log \log n}{\log \log \log n + C} - 1
\end{equation}
for $n > 25$, where $C > 0$ is an effectively computable constant depending
on~$a$ and~$b$. This bound leads to the following lemma.

\begin{lemma} \label{lem:sp-power}
If $p \neq q$ are primes, then
$$
  \lim_{k \to \infty} s_p \mleft( q^k \mright) = \infty.
$$
In particular, there exists a positive integer $L_{p,q} > \log_q p$ such that
$$
  s_p \mleft( q^k \mright) \geq p \quad (k \geq L_{p,q}).
$$
\end{lemma}

\begin{proof}
By taking $a=p$, $b=q$, $k \geq 5$, and $n=q^k > 25$,
we derive from~\eqref{eq:sum-digits} that
$$
  s_p(q^k) > \frac{\log(k \log q)}{\log \log(k \log q) + C} - 2 =: f(k)
$$
with some constant $C > 0$ depending on~$p$ and~$q$. Since $f(k)$ is strictly
increasing for all sufficiently large $k$, we infer that
$\lim_{k \to \infty} s_p \mleft( q^k \mright) = \infty$.
Therefore, there exists a positive integer $L_{p,q}$
such that $s_p \mleft( q^k \mright) \geq p$ for $k \geq L_{p,q}$.
On the other hand, since $s_p(m) = m$ for $0 \leq m < p$,
we have that $s_p(q^k) < p$ for $1 \leq k < \log_q p$,
implying that $L_{p,q} > \log_q p$, as claimed.
\end{proof}

\begin{proof}[Proof of Theorem~\ref{thm:Dn-ratio1}]
If $n=1$, then $\DD_{1} / \DD_{2} = 1$. By \eqref{eq:Dn-odd},
if $n \geq 3$ is odd and $n+1$ is not a power of $2$, then $\DD_{n}$ and
$\DD_{n+1}$ are both even. Since by \eqref{eq:Dn-formula} they are squarefree,
$\DD_{n} / \DD_{n+1}$ must be odd.

Likewise, if $n = 2^k - 1$ for some $k \geq 2$, then $\DD_{n} / \DD_{n+1}$
must be twice an odd number. If an odd prime $p$ divides $\DD_{n}$, then
$s_p(n) \geq p$ by \eqref{eq:Dn-formula}. Since $p \nmid 2^k = n+1$, we infer
by \eqref{eq:sp-values-ndiv} that $s_p(n+1) > p$. Hence by \eqref{eq:Dn-prod}
the prime $p$ also divides $\DD_{n+1}$, so indeed $\DD_{n} / \DD_{n+1} = 2$.

Now, let $n = 2^\ell p^k - 1$ with $p$ an odd prime and $k, \ell \geq 1$.
Then we have $\rad(n+1) = 2p$, and by \eqref{eq:Dn-odd} that $\DD_{n}$ and
$\DD_{n+1}$ are both even. Thus, $\DD_{n} / \DD_{n+1} \in \{ 1, p \}$
by Corollary~\ref{cor:Dn-relation} part~(i). We consider two cases.

Case $1 \leq \ell < \log_2 p$:
Since $s_p(n+1) = s_p(2^\ell) < p$, we infer that $p \nmid \DD_{n+1}$
by \eqref{eq:Dn-prod} implying $\DD_{n} / \DD_{n+1} = p$.

Case $\ell > \log_2 p$:
Lemma~\ref{lem:sp-power} implies a constant $\LB_p := L_{p,2} > \log_2 p$
such that $s_p(n+1) = s_p(2^\ell) \geq p$ for all $\ell \geq \LB_p$.
Hence $p \mid \DD_{n+1}$ by \eqref{eq:Dn-prod} and $\DD_{n} / \DD_{n+1} = 1$.
This proves the theorem.
\end{proof}

\begin{proof}[Proof of Theorem~\ref{thm:Dn-ratio2}]
It is shown in \cite[Thm.~4]{Kellner&Sondow:2017} that $\DB_{n}$ is even and
squarefree for all $n \geq 1$. (This also follows from \eqref{eq:DB-formula}
for even $n \geq 2$, since $2 \mid \D_n$, and from \eqref{eq:DB-formula2} for
odd $n \geq 1$, since $2 \mid \rad(n+1)$, all terms in question being squarefree.)
Hence if $n \geq 2$ is even, so that $\DB_{n+1} \mid \DB_{n}$, then the quotient
must be odd.

Let $p$ be an odd prime. If $n = p^k-1$ for some $k \geq 1$, then we have
$\rad(n+1) = p$ and $s_p(n+1) = s_p(p^k) = 1 < p$. Thus $p \nmid \DD_{n+1}$
by \eqref{eq:Dn-prod}. Since $n$ is even, we have $\DB_{n+1} = \DD_{n+1}$ by
\eqref{eq:DB-formula} and so $p \nmid \DB_{n+1}$. By Corollary~\ref{cor:Dn-relation}
part~(i) we finally obtain $\DB_{n} / \DB_{n+1} = p$.

Now, let $p \neq q$ be odd primes and $n = p^k q^\ell - 1$ with $k, \ell \geq 1$.
We then have $\rad(n+1) = pq$ and by Corollary~\ref{cor:Dn-relation} part~(i)
that $\DB_{n} / \DB_{n+1} \in \{ 1, p, q, pq \}$.
Note that $s_p(n+1) = s_p(q^\ell)$ and $s_q(n+1) = s_q(p^k)$.
By Lemma~\ref{lem:sp-power} we define $\LB'_{p,q} := L_{q,p} > \log_p q$ and
$\LB''_{p,q} := L_{p,q} > \log_q p$.
We consider the following statements by using \eqref{eq:Dn-prod}:

If $1 \leq \ell < \log_q p$, then $s_p(q^\ell) < p$ and $p \nmid \DD_{n+1}$.
Otherwise, if $\ell \geq \LB''_{p,q}$, then $s_p(q^\ell) \geq p$ and $p \mid \DD_{n+1}$.

If $1 \leq k < \log_p q$, then $s_q(p^k) < q$ and $q \nmid \DD_{n+1}$.
Otherwise, if $k \geq \LB'_{p,q}$, then $s_q(p^k) \geq q$ and $q \mid \DD_{n+1}$.

All three cases of the theorem follow from the arguments given above,
since $\DB_{n+1} = \DD_{n+1}$. This completes the proof of the theorem.
\end{proof}


\section{Conclusion}

The numbers
$$
  \BP_{m,r}^n = m^n \left( B_n \mleft( \frac{r}{m} \mright) - B_n \right),
$$
shown by Almkvist and Meurman to be integers, play here a key role in proofs.
By their result, the polynomial $B_n(x)-B_n$, with an extra factor, takes integer
values at rational arguments $x = r/m$. In the present paper, the numbers
$\BP_{m,r}^n$ reveal their natural connection with the power sums of arithmetic
progressions $\SP_{m,r}^n(x)$. Moreover, the divisibility properties of
$\BP_{m,r}^n$ are important in attaining our results in Theorems~\ref{thm:main}
and~\ref{thm:main2}.


\section*{Acknowledgment}
We thank the anonymous referee for several suggestions.

\newpage


\end{document}